\documentclass[a4paper,12pt]{article}
\usepackage[english]{babel}
\usepackage{amssymb,amsmath,amsfonts,amsthm}

\newtheorem{theorem}{Theorem}
\newtheorem{lemma}{Lemma}

\newtheorem{proposition}{Proposition}
\theoremstyle{remark}
\newtheorem{remark}{Remark}

\def\indexes{\mathfrak{A}}
\def\Line{\mathfrak{L}}
\def\Plane{\mathfrak{P}}
\def\PG{\mathrm{PG}}
\def\dfeq{\stackrel{\mathrm{df}}{=}}
\def\supp{\mathrm{supp}}

\begin{document}
\title{Embedding in $q$-ary $1$-perfect codes and partitions%
\thanks{The work was financed by the Russian Science Foundation (grant No 14-11-00555)}%
}
\author{D. S. Krotov%
\thanks{Sobolev Institute of Mathematics, Novosibirsk, Russia; Novosibirsk State University, Novosibirsk, Russia. 
E-mail: krotov@math.nsc.ru}
, E. V. Sotnikova%
\thanks{Sobolev Institute of Mathematics, Novosibirsk, Russia. E-mail: lucernavesper@gmail.com}%
}
\maketitle
\begin{abstract}
  We prove that every 
  $1$-error-correcting code 
  over a finite field can be 
  embedded in a $1$-perfect code 
  of some larger length.
  Embedding in this context means 
  that the original code is a subcode 
  of the resulting $1$-perfect code 
  and can be obtained from it by repeated 
  shortening.
  Further, we generalize the results to partitions: 
  every partition of the Hamming space into
  $1$-error-correcting codes can be 
  embedded in a partition of a space 
  of some larger dimension 
  into $1$-perfect codes.
  For the partitions, the embedding length 
  is close to the theoretical bound for the general case
  and optimal for the binary case.
\end{abstract}

\section{Introduction}
The goal of the current work is to show that any (in general, nonlinear) code
that can correct at least one error is a subcode of a $1$-perfect code of some larger length.
Moreover, we prove a similar result for the partitions into $1$-error-correcting codes.
In \cite{AvgKro:embed}, the possibility to embed into $1$-perfect code was proven for binary codes.
In \cite{Rom:2012}, the ternary case was solved;
for the $q$-ary case with $q>3$, there are similar embedding results in \cite{Rom:2012},
but with the restriction that the embedded code is required to be at least $2$-error-correcting
(this restriction is very strict as almost all $1$-error-correcting codes are not $2$-error-correcting).
The reason of such restriction is that the method suggested in \cite{AvgKro:embed}
does not work in general case: the components that should be switched in the linear $1$-perfect code
to build the required subcode can intersect in the case $q>3$ (see Remark~\ref{rem:intersect}).
To avoid this problem, we suggest a modification of the method.

We will follow the convenient notation 
and line of reasoning from \cite{AvgKro:embed}
with three main differences.
At first, the key definition of a linear $i$-component
(in our notation, we will write a Greek letter instead of traditional $i$)
is now given in a usual form \cite{PheVil:2002:q}, 
while the required property 
is declared in Lemma~\ref{l:comp}
(the definition based on this property 
would look complicate in the $q$-ary case).
At second, the formulation of the crucial proposition,
which is essentially the main and largest part
of the proof of the main theorem, 
is different from the crucial lemma in the binary case
(as was noted above, the last one does not work in the general case, see also Remark~\ref{rem:intersect}).
At third, we add the theorem about embedding partitions,
which is new for all $q$, including $q=2$.

\section{Notation and definitions}
Over the article, we will use the following notation.
\begin{itemize}
\item $F$  denotes the Galois field $\mathrm{GF}(q)$ of order $q$.
\item $F^m$ is the set of $m$-tuples over $F$, considered as a vector space over $F$. 
The elements of $F^m$ are denoted by Greek letters.
\item $\indexes$ consists of $m$-tuples from $F^m$ with the first nonzero element equal to $1$.
\item The intersection of $\indexes$ with a $2$-dimensional subspace of $F^m$ will be referred to as a \emph{line}.
The cardinality of every line is $q+1$. The set of lines together with the set of the \emph{points} $\indexes$
form an incidence structure, known as the \emph{projective geometry} $\PG(m-1,q)$.
\item The intersection of $\indexes$ with a $3$-dimensional subspace $F^m$ will be referred to as a \emph{plane}.
\item $n\dfeq\frac{q^{m}-1}{q-1}$.
\item $\Pi=\{\pi^{(1)},\,\ldots,\,\pi^{(m)}\}=\{(1\,0\,\ldots\,0),\,\ldots,\,(0\,\ldots\,0\,1)\}$ 
      is the natural basis in $F^m$.
\item The elements of $F^n$ will be denoted by overlined letters with the coordinates indexed by the elements of $\indexes$. We assume that the first $m$ coordinates have the indexes $\pi^{(1)},\,\ldots,\,\pi^{(m)}$, 
while the other $n-m$ coordinates are ordered in some arbitrary fixed way.
\item $\{\bar e^{(\delta)}\}_{\delta \in \mathfrak{A}}$ is the natural basis in $F^n$, 
      herewith $\bar e^{(\pi^{(\delta)})}=(\pi^{(\delta)},0^{n-m})$.
\item For any $\alpha = (\alpha_1,\,\ldots,\,\alpha_m) \in F^m$, 
we define $\bar \alpha \dfeq(\alpha,\, 0^{n-m}) \in F^n$; 
moreover $\bar \alpha = \sum_{i=1}^m \alpha_i\bar e^{(\pi^{(i)})}$.
\item The \emph{Hamming distance} $d(x,\,y)$ is the number of positions in which vectors $x$, $y$ from the same space differ. 
\item The \emph{neighborhood} $\Omega(M)$ of a set $M\subset F^n$ is the set of the vectors at distance at most $1$ from $M$.
\item A set $C \subset F^m$ is called a \emph{$1$-code} if the neighborhoods of the codewords are disjoint.
\item A $1$-code $P \subset F^n$ is called a \emph{$1$-perfect code} if $\Omega(P)=F^n$.
\item The \emph{Hamming code} $\mathcal H_m$ of length $n$ is defined as the set of vectors $\bar c \in F^n$ satisfying the following equation:
\begin{equation}\label{eq:hammingcode}
\sum_{\alpha\in\indexes}c_\alpha\alpha = 0^m.
\end{equation}
\item $\supp(c)=\{\delta\in \mathfrak{A} \mid c_\delta \ne 0\}$.
\item $T \dfeq \{c\in \mathcal H_m \,\bigl|\,\, |\supp(c)|=3\}$. 
      The elements of $T$ are called triples.
\item $T_{\delta} \dfeq\{c \in T \mid c_{\delta}=1\}$.
\item The \emph{linear $\delta$-component} $R_{\delta}$ is defined as the linear span $<T_{\delta}>$.
By an \emph{$\delta$-component of the Hamming code}, we will mean any coset of the linear $\delta$-component
that is a subset of the Hamming code.
\end{itemize}
\section{Preliminaries}\label{s:pre}
\begin{lemma}\label{l:compneighborhood}
For any $\bar z \in F^n$ it holds that $\Omega(R_\delta+\bar z)=\Omega(R_\delta+\bar z+ \mu \bar e^{(\delta)})$ for all $\mu \in F$.
\end{lemma}
\begin{proof}
Without loss of generality it is enough to prove the statement for $\bar z=0^n$. It is shown in \cite{PheVil:2002:q} that $(\mathcal H_m \setminus R_{\delta})\cup(R_{\delta}+\mu\bar e^{(\delta)})$ is a 1-perfect code for all $\mu \in F$. From the definition of the 1-perfect code it follows that the neighborhoods of the sets $R_{\delta}$ and $R_{\delta}+\mu\bar e^{(\delta)}$ are equal. So the statement of Lemma is true.
\end{proof}

\begin{lemma}\label{l:comp}
Let $\delta \in \indexes$. 
Every word $\bar c$ from $R_\delta$ satisfies the relation
\begin{equation}\label{eq:for-comp}
  \sum_{\alpha\in \Line} c_\alpha l(\alpha) = 0^m
\end{equation}
for all linear functions $l$ from $F^m$ to $F$ such that $l(\delta)=0$ and all lines $\Line$ containing $\delta$.
\end{lemma}
\begin{proof}
Since $R_\delta$ is a subset of the Hamming code, each of its elements $\bar c$ satisfies (\ref{eq:hammingcode}).
Then 
\begin{equation}\label{eq:partial}
  \sum_{\alpha\in \indexes} c_\alpha l(\alpha) = 0^m
\end{equation}
holds for any linear function $l$. Now assume $l(\delta)=0$
and consider a line $\Line$ containing $\delta$.
Then, the support of every triple from $T_{ \delta}$ either is included in $\Line$ or intersect with $\Line$ in only one element $\delta$. In the last case, (\ref{eq:for-comp}) is trivial; in the former case, 
it trivially follows from (\ref{eq:partial}). 
Since the required relation holds for every element of $T_{ \delta}$, we see from linearity 
that it holds for the linear span of $T_{\delta}$, i.e., for $R_{\delta}$.
\end{proof}

\begin{lemma}\label{l:comp2}
Let $\delta$, $\kappa \in \indexes$. 
Every element $\bar c$ of $< R_\delta, R_\kappa >$  satisfies the relation
\begin{equation}\label{eq:for-comp2}
  \sum_{\alpha\in \Plane} c_\alpha l(\alpha) = 0^m
\end{equation}
for all linear functions $l$ from $F^m$ to $F$ such that $l(\delta)=l(\kappa)=0$ 
and all planes $\Plane$ containing $\delta$ and $\kappa$.
\end{lemma}
\begin{proof}
First consider the case $\bar c \in R_\delta$. 
Summarizing (\ref{eq:for-comp}) over the all lines containing $\delta$ and included in $\Plane$,
we get (\ref{eq:for-comp2}).
So, the elements of $R_\delta$ and, similarly, the elements of $R_\kappa$ satisfy (\ref{eq:for-comp2}).
By the linearity, the elements of $< R_\delta, R_\kappa >$ do.
\end{proof}
\section{Embedding in a $1$-perfect code}\label{s:main}
\begin{proposition}\label{p:disjoint}
Assume that $\delta$ and $\kappa$ from $F^m$ 
both start with $1$ and the distance between them is at least
$3$. 
Then the $\delta$-component 
$R_{\delta}+\bar \delta-\bar e^{\delta}$ 
and the $\kappa$-component 
$R_{\kappa}+\bar \kappa-\bar e^{\kappa}$ are disjoint.
\end{proposition}
\begin{proof}
Consider the vector difference
$\bar c=(\bar \delta - \bar e^{(\delta)})
- (\bar \kappa - \bar e^{(\kappa)})$. 
It is sufficient to show that $\bar c \not\in <R_\delta, R_\kappa>$.
We will show that $\bar c$ does not satisfy (\ref{eq:for-comp2}).
Note that the first element of $\bar c$ is $0$,
and $c_{\pi^{(i)}} \ne 0$ if and only if $\delta_i \ne \kappa_i$.
Among the other coordinates (not from $\Pi$), $\bar c$
has exactly two nonzero positions, ${\delta}$ and ${\kappa}$.
Now consider some $i$ such that $c_{\pi^{(i)}} \ne 0$.
Note that $\pi^{(i)}$, $\delta$ and $\kappa$ are linearly independent 
(indeed, a nontrivial linear combination of $\delta$ and $\kappa$ 
is either nonzero in the first position
or a multiple of $\delta-\kappa$, 
which has at least three nonzeros 
and thus cannot coincide with $\pi^{(i)}$);
hence there is a unique plane $P$
containing $\pi^{(i)}$, $\delta$ and ${\kappa}$.

Now we state that $\pi^{(i)}$, $\delta$ and ${\kappa}$ are the only points of 
$P$ in which $\bar c$ is not equal to zero. 
Indeed, assume
that $\beta = h \pi^{(i)}+ a  \delta + b {\kappa} \in \indexes$.
If $a+b \ne 0$ then $\beta_1 \ne 0$ and thus either $\beta \in \{\delta,\kappa\}$ or 
$c_\beta = 0$ holds.
If $a+b = 0$ then 
$a  \delta + b {\kappa}= a  (\delta - {\kappa})$ and thus,
by the hypothesis of the proposition, this combination has at list three nonzero positions. 
In this case, 
$\beta$ has at least two nonzero positions, and thus does not belong to $\Pi$.
Hence, $c_\beta = 0$.

Then we consider a linear function $l$ such that $l(\delta)=l(\kappa)=0\ne l(\pi^{(i)})$ and see
that (\ref{eq:for-comp2}) cannot hold as it has only one nonzero summand, $\alpha = \pi^{(i)}$.
\end{proof}

\begin{remark}\label{rem:intersect} 
The hypothesis that both $\delta$ and $\kappa$ start with $1$ 
is necessary in Proposition~\ref{p:disjoint}
for $q>3$. 
For example let us consider $\delta = (1,1,1)$ and 
$\kappa = (t, t^2, t^2) = t (1,t,t) = t \gamma$,
where $t^2$ is different from $1$ and $t$ (so, $q\ge 4$).
Then the vectors $\overline \delta$ and $\overline \kappa$ are at distance $1$ 
from the $\delta$-component $R_{\delta}+\overline \delta-\overline e^{\delta}$
and the $\gamma$-component $R_{\gamma}+\overline \kappa-t\overline e^{\gamma}$ of the Hamming code.
It is easy to see that the nonzero coordinates 
$\pi^{(1)}$, $\pi^{(2)}$, $\pi^{(3)}$, $\delta$ and $\gamma$
of the difference 
$\overline c=(\overline \delta - \overline e^{(\delta)})
- (\overline \kappa - \overline t e^{(\gamma)})$ belong to the same plain.
Hence, since this difference is from the Hamming code, 
we see that it satisfies (\ref{eq:for-comp2}).
It is not difficult to conclude that the corresponding components intersect.
\end{remark}

\begin{theorem}\label{th:code}
Let $C \subset F^{m-1}$ be a 1-code. Define $\dot C\dfeq\{(1,x)\mid x\in C\}$. 
Then the following set
$$P(C)\dfeq\left.\Biggl(\mathcal H_m \right\backslash  \bigcup_{\delta\in \dot C} (R_{\delta}+\bar\delta-\bar e^{(\delta)}) \Biggr)\cup\Biggl(\bigcup_{\delta\in \dot C}(R_{\delta}+\bar\delta)\Biggr)$$
is a 1-perfect code in $F^n$, within 
\begin{equation}\label{eq:origincode}
C=\{x \in F^{m-1} \mid (1,\, x,\, 0^{n-m}) \in P(C)\}.
\end{equation}
\end{theorem}
\begin{proof} 
It is clear that $\bar\delta-\bar e^{(\delta)}\in \mathcal H_m$ for all $\delta\in \indexes$, which means $R_{\delta}+\bar\delta-\bar e^{(\delta)}\subset \mathcal H_m$ for all $\delta$. According to Proposition~\ref{p:disjoint} the sets $R_{\delta}+\bar\delta-\bar e^{(\delta)}$ are mutually disjoint for all $\delta \in \dot C$. 
As they are subsets of a $1$-perfect code, their neighborhoods are also mutually disjoint. From Lemma \ref{l:compneighborhood} we see that $P(C)$ is a $1$-perfect code.

To prove (\ref{eq:origincode}), we first note that 
$\bar c=(\alpha,0^{n-m})\in \mathcal H_m$ implies $\alpha=0^m$, 
which follows from the definition of Hamming code. 
Finally, we need to show that if for some $x \in F^{m-1}$ we have $(1,x,0^{n-m})\in R_{\delta}+\bar \delta$, then $(1,x)=\delta$. Indeed, if $(1,x,0^{n-m})\in R_{\delta}+\bar \delta$ then $(1,x,0^{n-m})-\bar \delta \in R_\delta \subset \mathcal H_m$, which only holds for $(1,x)=\delta$. This completes the proof.
\end{proof}

\section{Partitions}\label{s:part}
\begin{theorem}\label{th:to-part}
 Let $(C_1,\ldots,C_k)$ 
 be a partition of $F^{s}$ 
 into $1$-codes. 
 Then there is a partition 
 $(P_1,\ldots,P_{q^{s+1}})$ of $F^n$ 
 into $1$-perfect codes 
 of length $n=(q^{s+1}-1)/(q-1)$
 such that for all $j=1,\ldots,k$,
\begin{equation}\label{eq:part}
C_j=\{x \in F^{s} \mid (1,\, x,\, 0^{n-s-1}) \in P_j\}.
\end{equation} 
\end{theorem}
\begin{proof}
 Put $m=s+1$. Let for all $\alpha$ from $F^m$, 
 $H_{\alpha}$ be the coset of the Hamming code that contains $\bar\alpha$;
 so, $\{H_{\alpha}\}_{\alpha \in F^m}$ is a partition of $F^n$.
 Let us choose $k$ distinct vectors $y_1$, \ldots, $y_k$ from $F^{s}$, 
 and denote $\alpha_j=(0, y_j)$, $j=1,\ldots,k$.
Using Theorem~\ref{th:code}, 
we replace the code $H_{\alpha_j}$ by $P(C_j-y_j)+\bar\alpha_j$.
Then, this code will be the $j$s element $P_j$ of the constructed partition; 
readily, (\ref{eq:part}) is straightforward from (\ref{eq:origincode}).
It remains to replace the other cosets of the Hamming code to get a partition. 
According to the definition of $P_j$, 
it intersects with the following cosets of the Hamming code:
with $H_{\alpha_j}$ and, for every $x$ from $C_j$, 
with $H_{(1,x)}$, which has a common component with $P_j$.
Let $O_x$ be obtained from $H_{(1,x)}$ by removing this component, and by including the corresponding component
(the one that is not in $P_j$) of $H_{\alpha_j}$.
Now we see that the $|C_j|+1$ codes $H_{\alpha_j}$ and $O_x$, $x\in C_j$, 
are mutually disjoint and 
$$
P_j \cup \bigcup_{x\in C_j} O_x = H_{\alpha_j}\cup \bigcup_{x\in C_j} H_{(1,x)}.
$$
Then, the codes $P_j$, $j=1,\ldots,k$, together with the codes $O_x$, $x\in F^{m-1}$, and the codes
$H_{\alpha}$, where $\alpha$ does not start with $1$ and is different from all  $\alpha_j$, $j=1,\ldots,k$,
form a partition of $F^n$. As was noted above, (\ref{eq:part}) holds.
 \end{proof}

Note that, since the number $k$ of codes 
in the original partition can be rather large, 
up to $q^{s}$, the length $n$ for which 
it is possible to construct the embedding cannot be small too:
the number $(q-1)n+1$ of perfect codes 
in the resulting partition cannot be smaller than $q^s$.
So, $n \ge \frac{q^s-1}{q-1}$, 
and we see that our construction gives an embedding with ``almost''
minimal length $\frac{q^{s+1}-1}{q-1}$. 
Using the same approach and based on the results of 
\cite{AvgKro:embed} and \cite{Rom:2012}, 
one can construct an embedding of minimal length for the cases
$q=2$ and $q=3$, respectively.

Finally, we note that Theorem~\ref{th:to-part} is the most general known formulation
that generalize Theorem~\ref{th:code} and, in particular, the result of \cite{AvgKro:embed}
(putting aside small increasing of the embedding length) and some results of \cite{Rom:2012}.
As noted in \cite{AvgKro:embed}, 
the classical results \cite{Treash:TStoSTS} and \cite{Ganter:QStoSQS}
about embedding in Steiner triple systems and Steiner quadruple systems respectively
can also be treated as partial cases of this theorem.

\section{Acknowledgements}
This research was financed by the Russian Science Foundation (grant No 14-11-00555).


\providecommand\href[2]{#2} \providecommand\url[1]{\href{#1}{#1}}
  \providecommand\bblmay{May} \providecommand\bbloct{October}
  \providecommand\bblsep{September} \def\DOI#1{{\small {DOI}:
  \href{http://dx.doi.org/#1}{#1}}}\def\DOIURL#1#2{{\small{DOI}:
  \href{http://dx.doi.org/#2}{#1}}}\providecommand\bbljun{June}

\end{document}